\documentclass{amsart}
\usepackage{amssymb,latexsym}
\usepackage{amsmath,amscd}
\usepackage{graphicx}
\usepackage{fullpage}
\usepackage{url}
\DeclareGraphicsExtensions{.pdf,.png,.jpg,.eps}
\usepackage{epstopdf}
\usepackage{hyperref}
\usepackage{cleveref}

\newtheorem{theorem}{Theorem}
\newtheorem{lemma}{Lemma}
\newtheorem{prop}{Proposition}
\newtheorem{case}{Case}
\newtheorem{cor}{Corollary}
\newtheorem{subcase}{Case}
\numberwithin{subcase}{case}

\theoremstyle{definition}
\newtheorem*{definition}{Definition}
\newtheorem*{remark}{Remark}

\newcommand\C{\mathbb{C}}

\newcommand\Z{\mathbb{Z}}
\newcommand\R{\mathbb{R}}

\begin{document}

\title{Linear Triangle Dynamics: The Pedal Map and Beyond. }
\markright{Linear Triangle Dynamics}
\author{Claire Castellano and Corey Manack}

\maketitle

\begin{abstract}
We present a moduli space for similar triangles, then classify triangle maps $f$ that arise from linear maps on this space, with the well-studied pedal map as a special case. Each linear triangle map admits a Markov partition, showing that $f$ is mixing, hence ergodic.
\end{abstract}

\section{Introduction}
\label{Intro}
Because our main result exploits the symmetries a triangle can possess, we collect a few definitions. For us, a {\emph{triangle} $T=z_1z_2z_3$ is an ordered triple $(z_1,z_2,z_3)$ of distinct points in the complex plane $\C$, and $T$ is \emph{flat} if $z_1,z_2,z_3$ are collinear. Let $\mathcal{T}$ be the collection of triangles. Our convention is to draw edges cyclically from $z_i$ to $z_{i+1}$, (indices modulo $3$). Two triangles are \emph{similar} if the sets of edge lengths are proportional. The \emph{normalized principal angle} between complex numbers $z_1,z_{2}$ is given by \[\theta_{z_1,z_{2}} = \frac{1}{\pi}\arccos\left(\frac{\operatorname{Re}(z_1\overline{z_{2}})}{|z_1||z_{2}|}\right),\quad 0\leq\theta_{z_i,z_{i+1}}\leq 1\] and the (normalized) \emph{interior angle} of a triangle at vertex $z_i$ is $\alpha_i:=\theta_{z_{i+2}-z_i,z_{i+1}-z_i}$. The $\emph{shape}$ of $T$ is the ordered triple of interior angles $(\alpha_1,\alpha_2,\alpha_3)$ with $\alpha_1+\alpha_2+\alpha_3 =1$. By the law of sines, two triangles $T_1,T_2$ are \emph{similar} if the shape of $T_1$ is equal to the shape of $T_2$, up to permutation of vertices. If we let $S$ be the group generated by affine transformations $z \mapsto az+b, a,b\in\C, a\neq 0$ and the group $\Sigma_3$ of permutations on $3$ letters, then the diagonal action of $S$ on $\C^3$ partitions $\mathcal{T}$ into collections of similar triangles. Call $[T]\in\mathcal{T}/S$ the \emph{similarity class} of $T$. By means of sterographic projection, we append to $\mathcal{T}$ degenerate triangles with a vertex at infinity. If we declare the interior angle at infinity to be $0$ then the interior angles of degenerate triangles sum to $1$ and the $S$-action formulation of smilarity extends to these triangles as well. We say that a \emph{triangle map} is any function $f:\mathcal{T}/S\to\mathcal{T}/S$. In other words, a function $g\colon \mathcal{T}\to\mathcal{T}$ is a triangle map if for every $s\in S$ there exists $u\in S$ such that $T_1 = sT_2$ implies $f(T_1)=uf(T_2)$. We introduce three sets to describe similaity classes in $\mathcal{T}/S$. 
Define $A$ be the plane \[A=\{(\alpha_1,\alpha_2,\alpha_3)\in\R^3\mid \alpha_1+\alpha_2+\alpha_3 =1\},\] 
$A_p\subset A$ be the set of interior angles \[A_p=\{(\alpha_1,\alpha_2,\alpha_3)\in [0,1]^3 \mid \alpha_1 \geq 0,\alpha_2 \geq 0,\alpha_3 \geq 0\},\]
and $D\subset A_p$ be the set of all ordered interior angles \[D=\{(\alpha_1,\alpha_2,\alpha_3)\in [0,1]^3 \mid \alpha_1\geq\alpha_2\geq\alpha_3\geq 0\}.\]
Notice that the interior angle function establishes a bijection $\phi: D \to \mathcal{T}/S$;  each shape $(\alpha_1,\alpha_2,\alpha_3)\in D$ describes a unique similarity class $[T]\in \mathcal{T}/S$ by assigning $\alpha_1$ to a largest interior angle, then $\alpha_2$ to the next largest, then $\alpha_3$. Thus, for each triangle map $f$ there exists a unique map $\overline{f}\colon D\to D$ such that the following diagram commutes:
\[\begin{CD}
D @>\overline{f}>> D\\
@VV\psi V    @VV\psi V \\
\mathcal{T}/S @>f>>  \mathcal{T}/S
\end{CD}\]
and we refer to a triangle map by either $f$ or $\overline{f}$. Our interest is in tracking the shape of $f^n([T])$ for various $f$, $n$, and $[T]$. 
The paper is organized as follows.  \Cref{sec:pedal} reviews the construction of pedal triangles, whose triangle map $P:D\to D$ we call the \emph{pedal map}. In \cref{sec:moduli} we show how to identify an element of $A$ to a shape in $D$ by describing a similarity class $[T]$ with nonprinciple angles. These identifications arise as the action of the wallpaper group $G=p6m$ on $A$, with $D$ homeomorphic $A/G$. Borrowing terminology from toral automorphisms, call $f \colon D\to D$ a \emph{linear triangle map} if $f$ is the quotient map of an invertible linear map $M:\R^3 \to \R^3$ that leaves $A$ invariant and preserves the identifications induced by $G$. Coordinatizing, we call $M$ an \emph{angle transition matrix} (ATM) of $f$. \Cref{sec:trimaps} contains our main theorem, which classifies all possible ATM's.

\begin{theorem}
\label{thm:circsym}[Classification of angle transition matrices]
Suppose $M:A\to A$ is an ATM. Then there exists $g\in G$ such that $gM$ is either:
\begin{enumerate}
\item (Type I) A circulant and symmetric matrix
\begin{equation}
\label{equ:goodmat}
\begin{bmatrix}
c_0 & c_1 & c_1 \\
c_1 & c_0 & c_1 \\
c_1 &  c_1 & c_0 \\
\end{bmatrix}
\end{equation} with $c_0\in\Z,c_1\in\Z, c_0+c_1+c_1=1$
\item (Type II)
\begin{equation}
\label{equ:badmat1}
T_w^{-1}\begin{bmatrix}
c_0/3 & c_1/3 & c_1/3 \\
c_1/3 & c_0/3 & c_1/3 \\
c_1/3 &  c_1/3 & c_0/3 \\
\end{bmatrix}
\end{equation} with $c_0\in\Z,c_1\in\Z, c_0+c_1+c_1=1$, $c_0$ congruent to $1$ mod $3$, and $T_w$ is the matrix that acts on $A$ by translation in the direction of $w=(1/3,1/3,-2/3)$.
\item (Type III)
\begin{equation}
\label{equ:badmat2}
\begin{bmatrix}
0 & k & -k \\
-k & 0 & k \\
k+1 &  -k+1 & 1 \\
\end{bmatrix}
\end{equation} with $k\in Z$
\end{enumerate}
\end{theorem}
\begin{figure}
\label{3types}
\begin{center}
\scalebox{.242}{\includegraphics{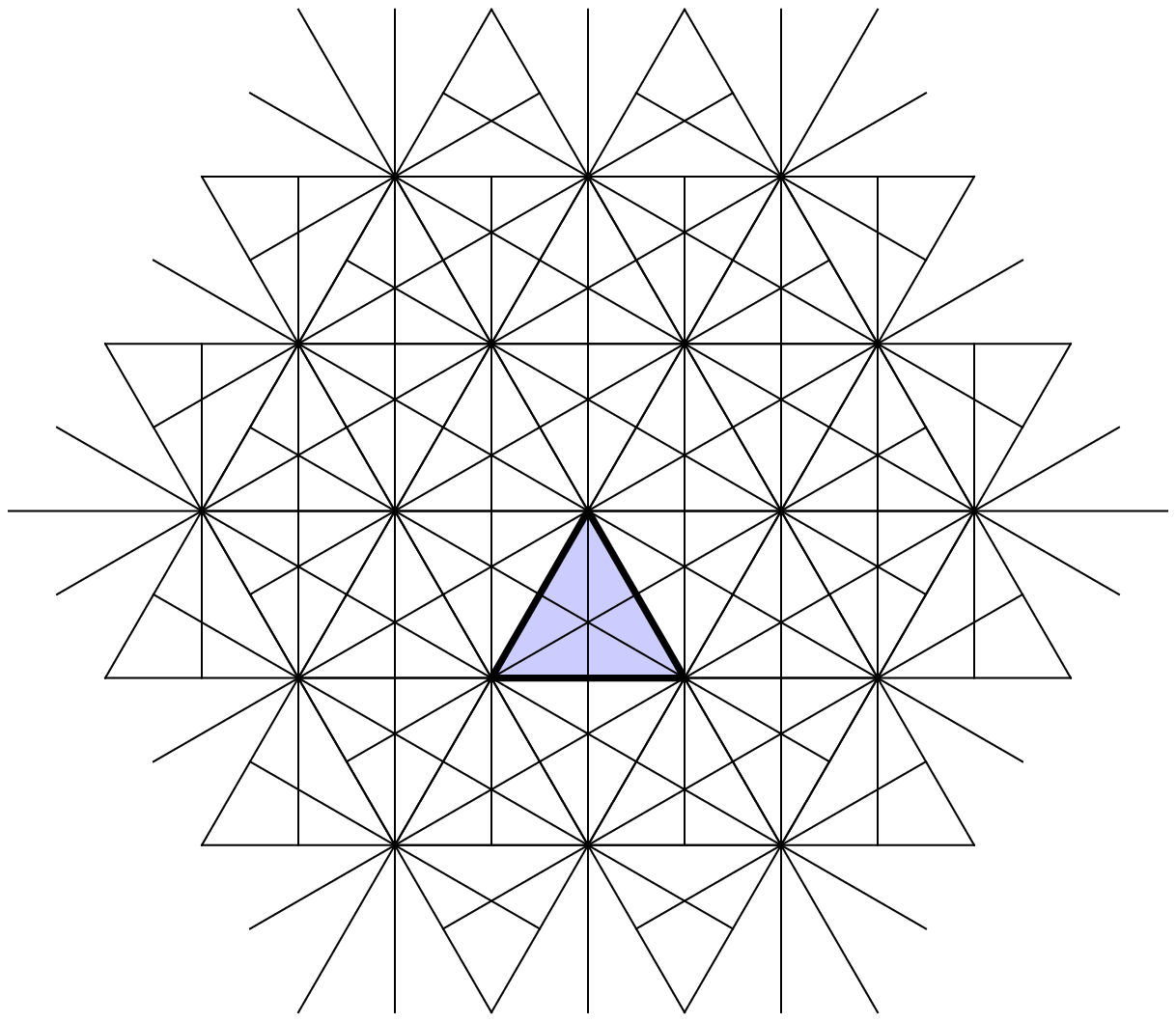}}
\scalebox{.242}{\includegraphics{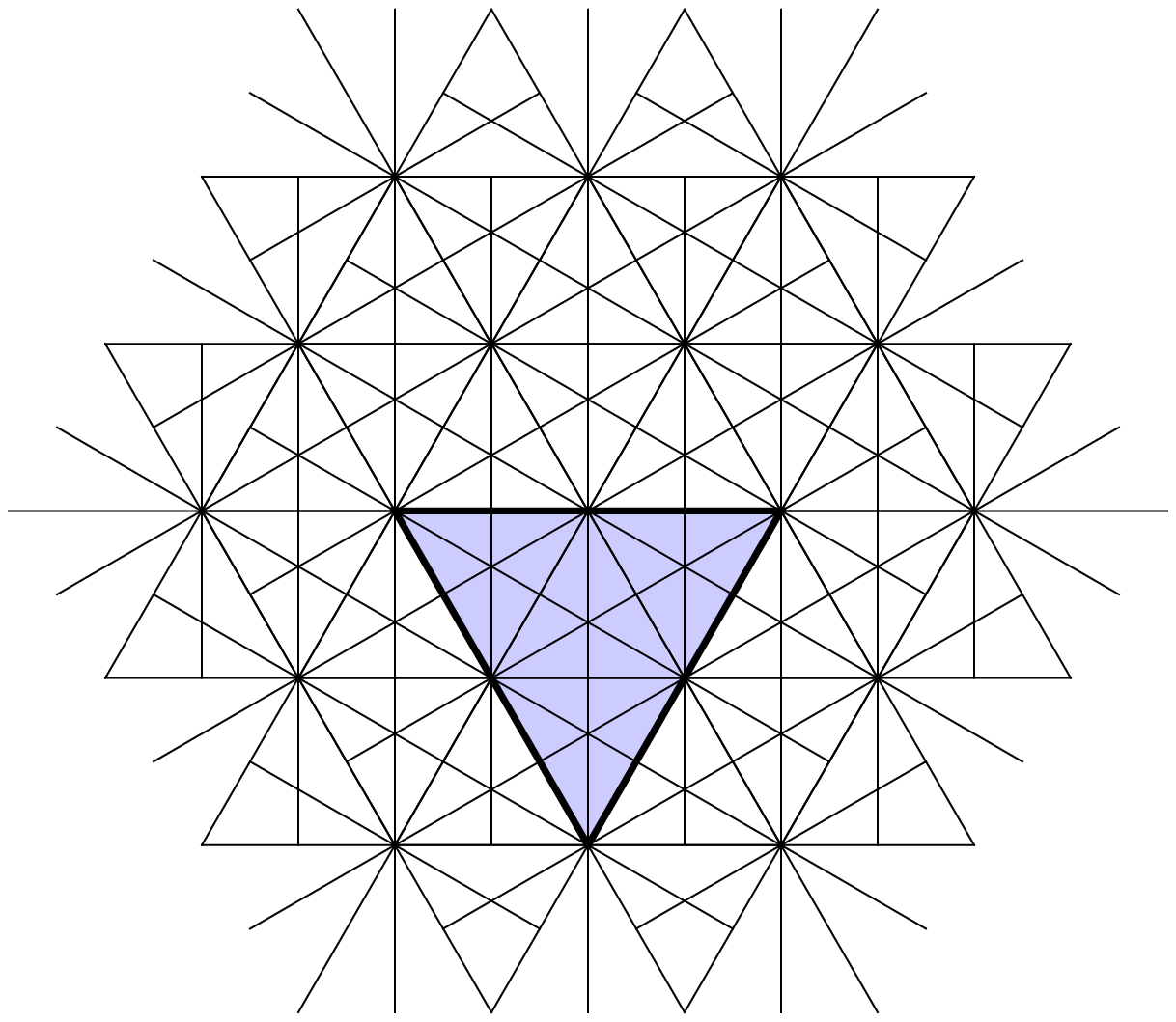}}
\scalebox{.242}{\includegraphics{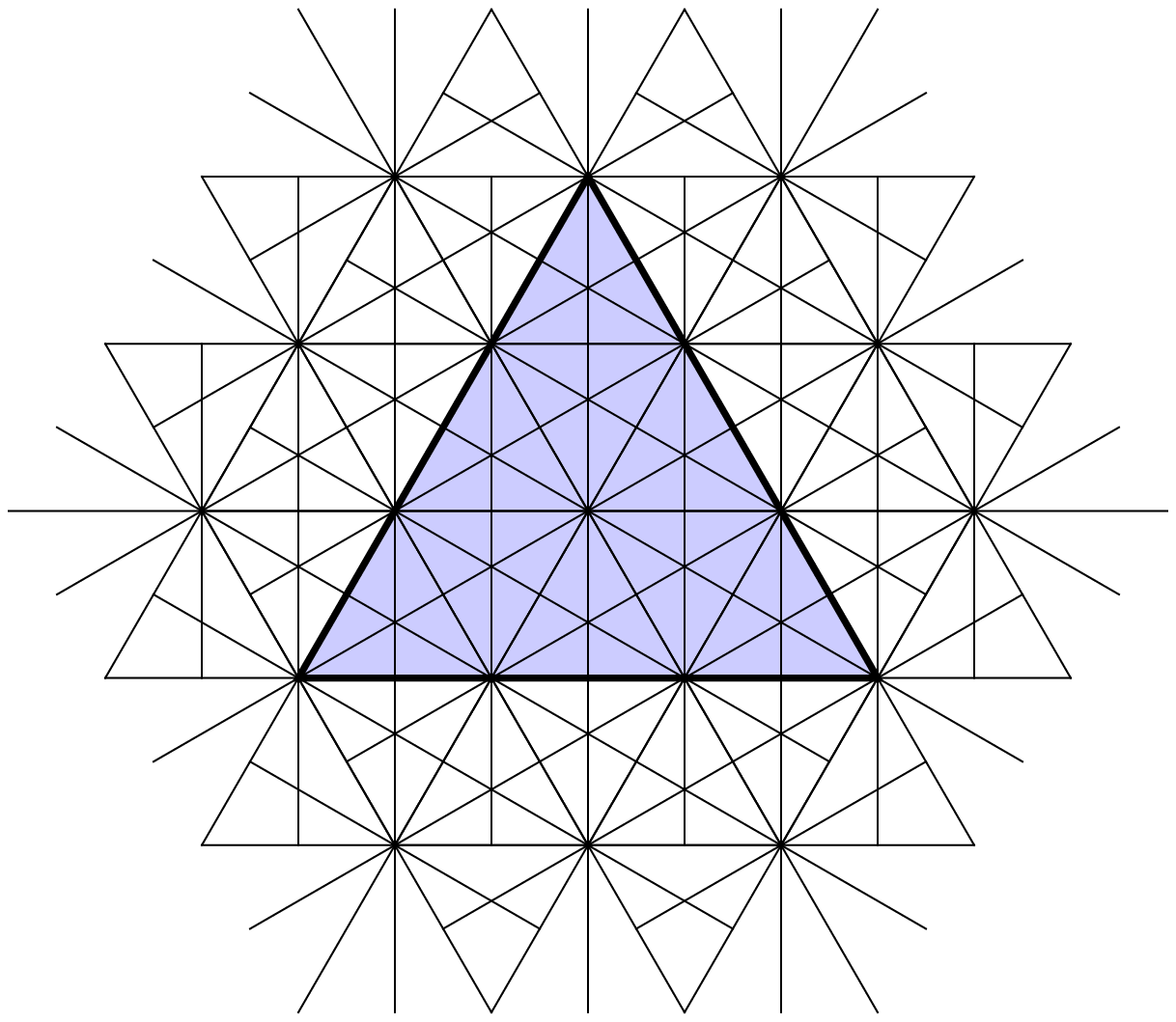}}
\scalebox{.242}{\includegraphics{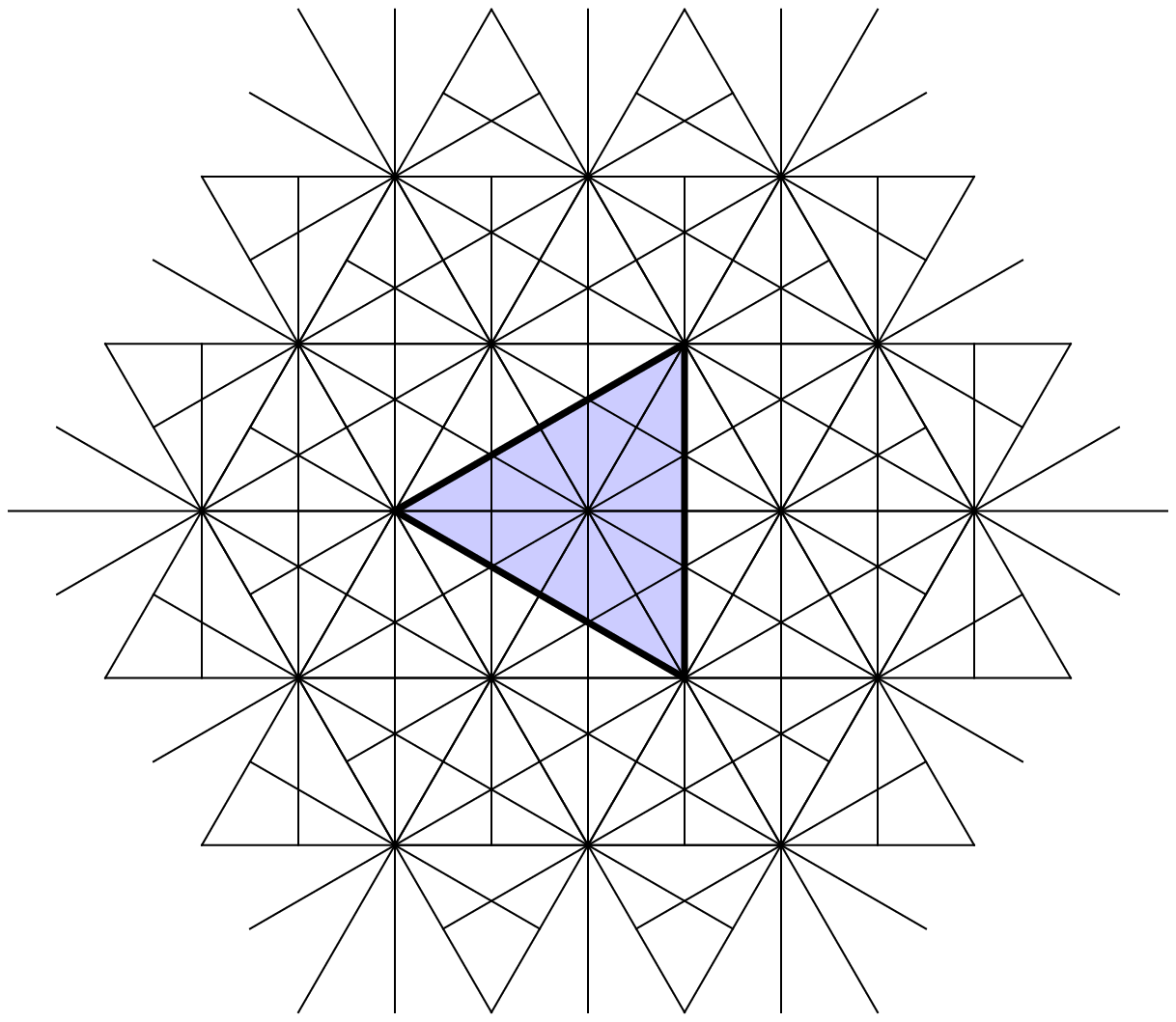}}
\end{center}
\caption{The four images display $MA_p$ (shaded in blue) in $A/G$ for various ATM's $M$. From left to right, $M$ is: the identity, the pedal map (Type I), Type II, and Type III }
\end{figure}
The three types are displayed in Figure \eqref{3types}. Many articles, some appearing in the \emph{Monthly}, have studied the measurable dynamics of the pedal map $P$, showing $P$ is ergodic \cite{La90}, mixing \cite{Un90}, and semiconjugate to a Bernoulli shift on $4$ symbols \cite{Un90},\cite{Al93}. Using the classification (\Cref{thm:circsym}) and well-known results about Markov partitions all such triangle maps are semiconjugate to a one-sided Bernoulli shift, hence mixing and ergodic (\Cref{thm:dynamics}). Since the (extended) pedal map is Type I linear (\Cref{thm:pedallin}), \Cref{thm:dynamics} includes the dynamics of $P$ as a special case. 
\section{The Pedal Map}
\label{sec:pedal}
Our motivating example begins with a reference triangle $T_0= x_1x_2x_3$.  For each vertex $x_i\in\C$ of $T_0$, consider the line $\overleftrightarrow{x_{i+1}x_{i+2}}$ passing through the other two points (indices modulo $3$). Let $y_1$ be the unique point of intersection between $\overleftrightarrow{x_{i+1}x_{i+2}}$ and its perpendicular through $x_i$. Adjoining edges between the feet of the three perpendiculars forms the first {\emph{pedal triangle}}  $T_1 = \Delta y_1y_2y_3$ from $T_0$.  Iterating $n$ times creates the $n$-th pedal triangle $T_n$ from $T_0$.  
Hobson correctly wrote down a formula for the (normalized) interior angles of $T_1$ in terms of $T_0$, namely, 
\begin{equation}
\label{Hobsonone}
\alpha_1 = 1-2\alpha_0,\beta_1 = 1-2\beta_1,\gamma_1 = 1-2\gamma_0
\end{equation}
when $T_0$ is acute, and
\begin{equation}
\label{Hobsontwo}
\alpha_1 = 2\alpha_0-1,\beta_1 = 2\beta_0,\gamma_1 = 2\gamma_0
\end{equation}
if $\alpha_0$ is obtuse, with similar formulas when $\beta_0$, $\gamma_0$ is obtuse. 
Notice that the formula degenerates for right triangles, when the feet of two perpediculars coincide. Hobson then went on to write down an formula for the interior angles $(\alpha_n,\beta_n,\gamma_n)$ of the $n$th pedal triangle $T_n$. Kingston and Synge \cite{KS88} recognized that Hobson's formula for $T_n$ was flawed, and in correcting it,\footnote[1]{Periodicity in iterated pedal triangles was observed for specific $T_0$ in \cite{Tu33} and \cite{Ta45}} proved that the sequence $\{(\alpha_n,\beta_n,\gamma_n)\}$ of interior angles is eventally periodic if angles of $T_0$ are rational and not dyadic. They introduced a simple way of parametrizing the shape of triangles by the set $A_p$, pictured in Figure \cref{KSAS}. 
\begin{figure}
\begin{center}
\label{KSAS}
\scalebox{.37}{\includegraphics{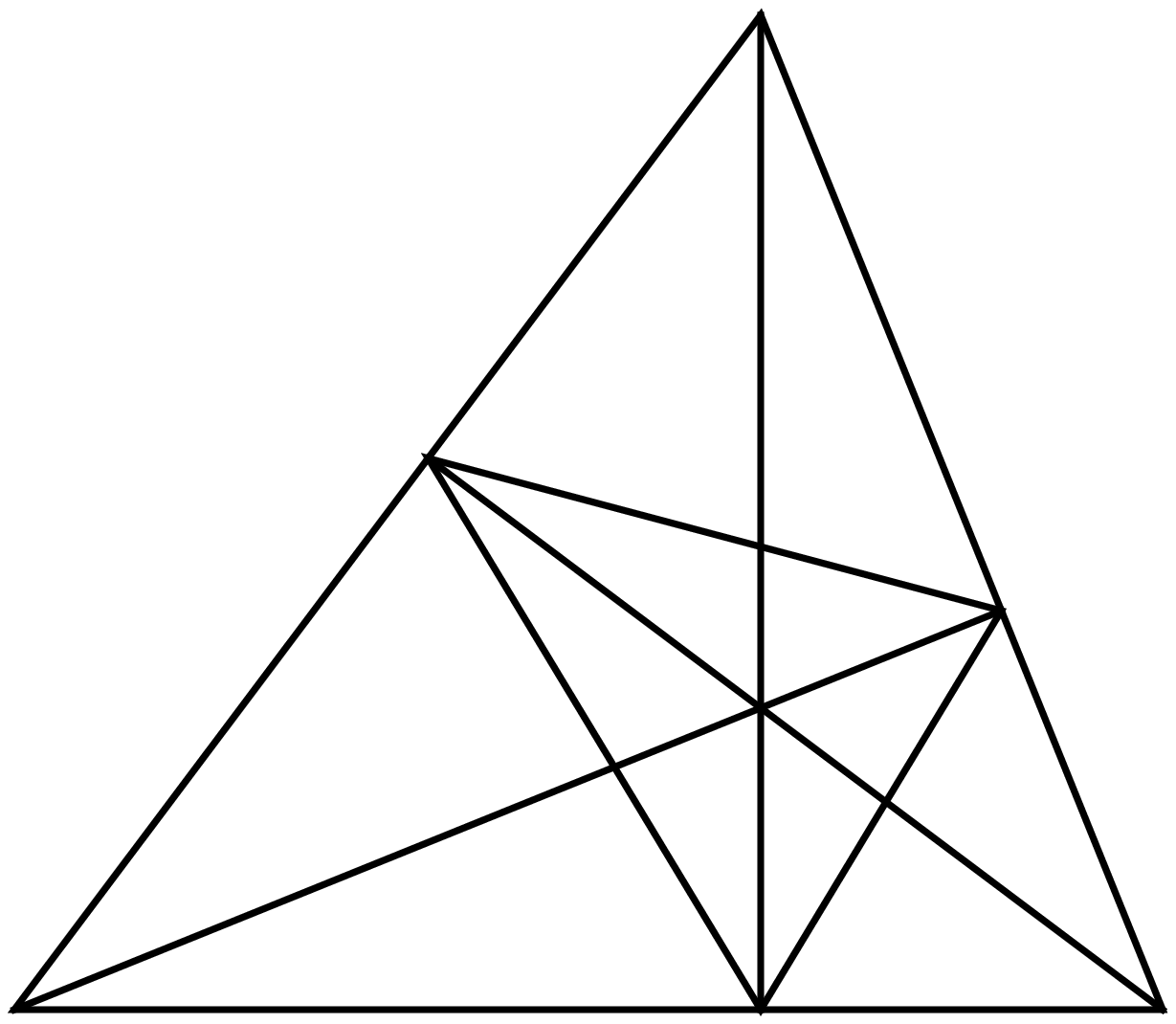}}\ \ \ \ \ \ \ \ \ \ \ \ \ \ \ \ \ \scalebox{.5}{\includegraphics{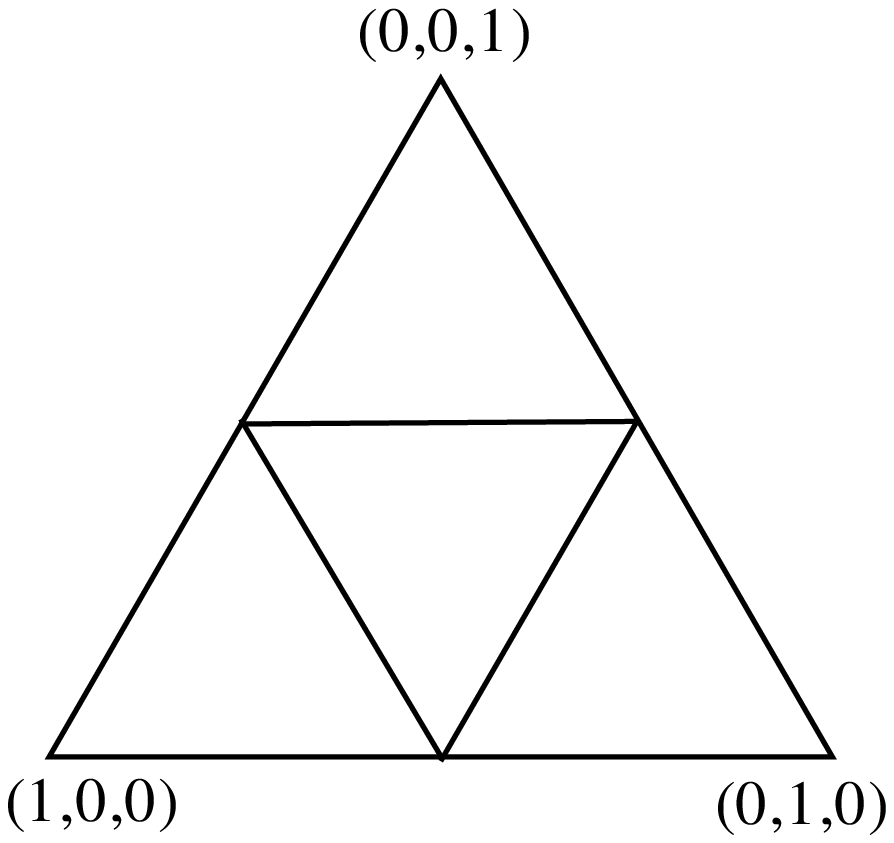}}
\end{center}
\caption{The figure on the left shows a reference triangle with first pedal triangle. The figure on the right is the space $A_p$ of interior angles, where the central triangle represents acute triangles and three outer triangles for obtuseness in $\alpha,\beta,\gamma$ respectively.} 
\end{figure}
The $3$ perpendiculars of $T$  meet at a point, called the  \emph{orthocenter} $O$ of $T$. Observe that Hobson's formulas define a piecewise four-to-one mapping $P:A_p\to A_p$, because the triangles $Oz_2z_3, z_1Oz_3, z_1z_2O$ also have first pedal triangle equal to $T$. Following \cite{KS88}, we call any one of these four preimages an \emph{ancestor} of $T$. Each nonflat $T_0$ has exactly one acute ancestor which we call the \emph{antipedal triangle} $T_{-1}$ of $T_0$. If we include similarity by permuting vertices then Hobson's formulas define a map $P: D\to D$ which we also call the pedal map. \\
Our intuition stems from the following two observations. First, the substitution $1=\alpha_n+\beta_n+\gamma_n$ into Hobson's formula yields, when $T_n$ is acute,
\begin{center}
\begin{tabular}{c}
$\alpha_{n+1}=-\alpha_n+\beta_n+\gamma_n$\\
$\beta_{n+1}=+\alpha_n-\beta_n+\gamma_n$\\ 
$\gamma_{n+1}=+\alpha_n+\beta_n-\gamma_n$\\
\end{tabular}\! = $\begin{bmatrix}
-1 & 1 & 1 \\
1 & -1 & 1 \\
1 & 1 & -1 
\end{bmatrix}\begin{bmatrix}
\alpha_n  \\
\beta_n  \\
\gamma_n  
\end{bmatrix}. $
\end{center}
so that the angles of $T_{n+1}$ are linear combinations of $T_n$. This was done implicitly in \cite[Section 3]{Ma10}. If we let $M$ be the matrix 
\[ M = \begin{bmatrix}
-1 & 1 & 1 \\
1 & -1 & 1 \\
1 & 1 & -1 
\end{bmatrix}\] then $M$ agrees with $P$ on the part of $A_p$ that describes acute triangles. While $M$ is invertible, $P$ is not. We leave it to the reader to verify that the angles $\alpha_{-1},\beta_{-1},\gamma_{-1}$ of $T_{-1}$ are given by
\[\begin{bmatrix}
\alpha_{-1}\\
\beta_{-1}\\
\gamma_{-1}
\end{bmatrix} =\begin{bmatrix}
0 & 1/2 & 1/2 \\
1/2 & 0 & 1/2 \\
1/2 & 1/2 & 0 
\end{bmatrix}\begin{bmatrix}
\alpha_0  \\
\beta_0  \\
\gamma_0  
\end{bmatrix} \]
Our second observation is that
\[ \begin{bmatrix}
0 & 1/2 & 1/2 \\
1/2 & 0 & 1/2 \\
1/2 & 1/2 & 0 
\end{bmatrix}^{-1}= \begin{bmatrix}
-1 & 1 & 1 \\
1 & -1 & 1 \\
1 & 1 & -1 
\end{bmatrix}. \]
\noindent and we see the inverse nature between pedaling and antipedaling in the angles of $T$ when $T$ is acute. Since $M^{-1}$ is a regular Markov chain with steady state vector $(1/3,1/3,1/3)$ and two eigenvalues of $-1/2$, the antipedal map is a contaction map on $A_p$, and $M^{-1}A_p$ excludes obtuse triangles. Indeed, for each $T_0$, $T_{-n}$ tends to equilateral as $n\to\infty$. Inverting, $M$ fixes $(1/3,1/3,1/3)$ and expands $A_p$ by a factor of $-2$, sending obtuse triples in $A_p$ to points in $A$ but outside $A_p$. Since $A_p$ describes all triangle shapes (up to permutation), we seek to identify triples in $A\setminus A_p$ to points in $A_p$ that describe the same similarity class. For $T_n$ obtuse in $\alpha_n$, again substitute $1=\alpha_n+\beta_n+\gamma_n$ into Hobson's formula to get
\begin{equation}
\label{Ralpha}
\begin{bmatrix}
\alpha_{n+1}\\
\beta_{n+1}\\
\gamma_{n+1}
\end{bmatrix} = \begin{bmatrix}
1 & -1 & -1 \\
0 & 2 & 0 \\
0 & 0 & 2 
\end{bmatrix}\begin{bmatrix}
\alpha_n  \\
\beta_n  \\
\gamma_n  
\end{bmatrix}=\begin{bmatrix}
-1 & 0 & 0 \\
1 & 0 & 1 \\
1 & 1 & 0 
\end{bmatrix}
\begin{bmatrix}
-1 & 1 & 1 \\
1 & -1 & 1 \\
1 & 1 & -1 
\end{bmatrix}
\begin{bmatrix}
\alpha_n  \\
\beta_n  \\
\gamma_n  
\end{bmatrix}.
\end{equation}
In section \cref{sec:moduli} we shall see that 
\begin{equation}
\label{equ:aref} R_{\alpha}:=\begin{bmatrix}
-1 & 0 & 0 \\
1 & 0 & 1 \\
1 & 1 & 0 
\end{bmatrix}
\end{equation}
arises naturally as an action on $A$, identifying points that describe the same similarity class in $\mathcal{T}/S$.\\
When studying the measurable dynamics of $P$, previous authors ignored those $[T_0]$ which were eventually flat since that subset of $A_p$ is a set of (normalized Lebesque) measure zero. Manning (\cite{Ma10}) augmented $A_p$ to include nonprincipal (at least one of $\alpha,\beta,\gamma <0$) and flat $T_0$ by identifying triples via orientation changes. He then showed that $P$ could be continously extended to flat triangles and found the limit point of the circumcenters of $T_n$ as $n\to\infty$ (\cite[Theorem 3.2]{Ma10}). In view of observation \eqref{Ralpha}, we recast and complete the identifications carried out in \cite{Ma10}, in the next section. 
\section{Moduli space for triangles}
\label{sec:moduli}
Loosely speaking, a moduli space is a collection of numerical ingredients that encode the salient features of geometric objects. We have already seen one moduli space for $\mathcal{T}/S$: the set $D$. Let $e_1,e_2,e_3$ be the standard coordinate vectors of $\R^3$. Whether we treat $v\in\R^3$ as a row or column vector will be clear from context. Define the equivalence relation $\sim$ on $A$ by writing $v=(\alpha_1,\alpha_2,\alpha_3),v'=(\alpha_1',\alpha_2',\alpha_3')$. Then \begin{center}
$v\sim v'$ if and only if, $\alpha_i\equiv -\alpha_{\sigma^{-1}(i)}' \mod 1$ or $\alpha_i\equiv -\alpha_{\sigma^{-1}(i)}' \mod 1$, $\sigma\in\Sigma_3$  
\end{center} 
and say that $v$ is a \emph{re-expression} of $v'$ if $v\sim v'$. The reader can verify that $\sim $ is an equivalence relation that identifies points in $A$ to shapes in $D$. We provide a geometric explanation of $\sim$. Let $T=z_1z_2z_3$ be a triangle with shape $(\alpha_1,\alpha_2,\alpha_3)\in A_p$. Permuting the vertices of $T$ by any $\sigma\in\Sigma_3$ will not change the shape of $T$. The induced $\Sigma_3$-action on $\R^3$, generated by
\begin{equation} 
\label{permmat}
P_{12}=\begin{bmatrix}
0 & 1 & 0 \\
1 & 0 & 0 \\
0 & 0 & 1
\end{bmatrix}, P_{13}=
\begin{bmatrix}
0 & 0 & 1 \\
0 & 1 & 0 \\
1 & 0 & 0
\end{bmatrix},
P_{23}=\begin{bmatrix}
1 & 0 & 0 \\
0 & 0 & 1 \\
0 & 1 & 0
\end{bmatrix}
\end{equation} leaves $A$ invariant, reflecting across the planes $x-y=0,x-z=0,y-z=0$. These planes intersect $A$ through the medians of $A_p$. Thus $\{\sigma D\mid \sigma\in \Sigma_3\}$ tiles $A_p$ by $6$ copies of $D$ and $v\sim w$ in $A_p$ if and only if $v=\sigma w$ for some $\sigma\in\Sigma_3$. 

Identifing points between $A_p$ and $A\setminus A_p$ requires two separate identifications. Let $(\alpha_1,\alpha_2,\alpha_3)\in A$. If we treat $\alpha_i\geq 0$ as rotation about the vertex $z_i$ (either clockwise of counterclockwise, depending on $T$) that takes the line $\overleftrightarrow{z_iz_{i+1}}$ to $\overleftrightarrow{z_iz_{i+2}}$ (indices modulo $3$), then $\alpha_i\pm k, k\in \Z$, describes another rotation with the same orientation as $\alpha_1$ that takes $\overleftrightarrow{z_iz_{i+1}}$ to $\overleftrightarrow{z_iz_{i+2}}$. Notice that the difference of two vectors in $A$ lie in the orthogonal hyperplane $(1/3,1/3,1/3)^{\perp}$. So  two vectors $v,v'$ in $A$ describe the same similarity class if $v-v'\in \Z\{e_1-e_2,e_2-e_3\}$; we recover interior angles by $\alpha_i\equiv \alpha_i' \mod 1$.  So $v\sim v'$. Notice further, if $(\alpha_1,\alpha_2,\alpha_3)\in A$ then rotation about $z_i$ by either $-\alpha_i$ (opposite orientation as $\alpha_i$) or $1-\alpha_i$ (same orientation as, but supplementary to, $\alpha_i$) takes $\overleftrightarrow{z_iz_{i+2}}$ to $\overleftrightarrow{z_iz_{i+1}}$. Ensuring that the angles sum to $1$, we make the identification \begin{equation}
\label{orient1}
(\alpha_1,\alpha_2,\alpha_3) \sim (-\alpha_1,1-\alpha_2,1-\alpha_3).
\end{equation} 
In this case, we recover the interior angles of $T$ by $\alpha_i\equiv -\alpha_i' \mod 1$.
Substituting $\alpha_1+\alpha_2+\alpha_3=1$ into \eqref{orient1},  
\begin{equation}
(\alpha_1,\alpha_2,\alpha_3) \sim (-\alpha_1,\alpha_1 + \alpha_3,\alpha_1+\alpha_2).
\end{equation} or  
\[\begin{bmatrix}
\alpha_1\\ 
\alpha_2\\
\alpha_3
\end{bmatrix} \sim 
\begin{bmatrix}
-1 & 0 & 0 \\
1 & 0 & 1 \\
1 & 1 & 0
\end{bmatrix}
\begin{bmatrix}
\alpha_1\\ 
\alpha_2\\
\alpha_3
\end{bmatrix}= R_{\alpha}\begin{bmatrix}
\alpha_1\\ 
\alpha_2\\
\alpha_3
\end{bmatrix}.\]
\begin{figure}
\label{CCCMAS}
\begin{center}
\scalebox{.45}{\includegraphics{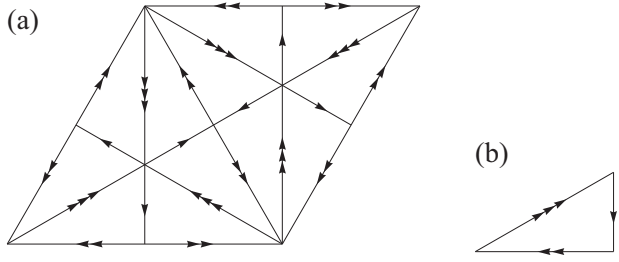}}
\end{center}
\caption{(a) An unfolded part of the moduli space $A/G.$ (b) fundamental domain $D$. }
\end{figure}
\noindent We have found the matrix $R_{\alpha}$ in \eqref{equ:aref}.  Making similar observations for $\alpha_2, \alpha_3$, we identify points in $A$ via the actions of
\begin{align}
\label{orientref}
R_{\alpha} &= \begin{bmatrix}
-1 & 0 & 0 \\
1 & 0 & 1 \\
1 & 1 & 0
\end{bmatrix} =
P_{23}\begin{bmatrix}
-1 & 0 & 0 \\
1 & 1 & 0 \\
1 & 0 & 1
\end{bmatrix} \\
\label{orientref2}
R_{\beta} &= \begin{bmatrix}
0 & 1 & 1 \\
0 & -1 & 0 \\
1 & 1 & 0
\end{bmatrix} =
P_{13}
\begin{bmatrix}
1 & 1 & 0 \\
0 & -1 & 0 \\
0 & 1 & 1
\end{bmatrix} \\
\label{orientref3}
R_{\gamma} &= \begin{bmatrix}
0 & 1 & 1 \\
1 & 0 & 1 \\
0 & 0 & -1
\end{bmatrix} =
P_{12}
\begin{bmatrix}
1 & 0 & 1 \\
0 & 1 & 1 \\
0 & 0 & -1
\end{bmatrix}
\end{align}
Geometrically, $R_{\alpha}$ is $P_{23}$ composed with a reflection in $\R^3$ across the plane that passes through $x=0$ and leaves $A$ invariant.
Notice that $R_{\alpha}^2=R_{\beta}^2=R_{\gamma}^2=id,$ and $R_{\beta}R_{\alpha},R_{\gamma}R_{\beta}$ translates $A$ by $e_1-e_2,e_2-e_3$ respectively. If we let $$R=\begin{bmatrix}
1 & 0 & 1 \\
0 & 1 & 1 \\
0 & 0 & -1
\end{bmatrix}$$ then $R$ acts on $A$ by reflection across the line $z=0$ in $A$. The reader can quickly verify that $R_{\alpha} = P_{23}P_{13}RP_{13}, R_{\beta}= P_{23}RP_{13}P_{12}, R_{\gamma}= P_{12}R$.  If we let $G$ be the group generated by the involutions $P_{13},P_{23},$ and $ R$, then each $g\in G$ leaves $A$ invariant. Each generator of $G$ reflects $D$ along one of its boundary edges and the $G$-action on $A$ is generated by reflections across lines that are intersections of $A$ with the planes 
\begin{equation}
\label{equ:refllines}
x\in\Z,y\in\Z,z\in\Z,x-y\in \Z,y-z\in \Z,x-z\in\Z
\end{equation} 
We have shown 
\begin{center}
$v\sim w$ if and only if $w=gv$ for some $g\in G$. 
\end{center}In other words, the $G$-orbit $Gv$ is the set of all re-expressions of $v\in A$. Under these identifications, the set of orbits $A/G$ is a quotient of $A$ by the rank $2$ lattice $\Lambda:=e_1+\Z\{e_1-e_2,e_2-e_3\}$, along with the $12$ identitfications from $v\sim R_{\alpha}v$, $v\sim \sigma v\ (\sigma\in \Sigma_3)$ as seen in Figure \cref{CCCMAS}(a). Every $G$-orbit $O$ has exactly one representative $p\in D$, so we call $D$ the {\emph{fundamental domain}} of  $A/G$. Observe $GD:=\{gp\mid g\in G, p\in D\} = A$ and $D$ is the (closed) $30-60-90$ triangle with vertices at $v_1=b=(1/3,1/3,1/3)$ and $v_2=(1,0,0)$, $v_3=(1/2,1/2,0)$. Thus the $G$-action on $A$ is that of the wallpaper group $p6m$.  For any $v\in A$, the {\emph{point group}} $H_v\subset G$ is the subgroup of $G$ that stabilizes $v$. As $GD=A$ and $gH_vg^{-1} = H_{gv}$,  $H_v$ is conjugate to exactly one $H_p$, $p\in D$. Since $G$ is generated by reflections along the $3$ boundary lines of $D$, $H_p$ is isomorphic to a dihedral group. Moreover, $|H_p|= 1,2,4,6$ or $12$ depending on whether $p$ lies, respectively, in the interior $D$, on an edge of $D$ (but not a vertex), or $p= v_3,b,$ or $v_2$. The next key lemma, the {\emph{local-point group property}} articulates the $G$-action on $A$. Denote by $B_{\epsilon}(v)$ the ball of radius $\epsilon$ centered at $v$.
\begin{theorem}
\label{thm:local prop}[Local point group property]
Let $G=p6m$ act on $A$ by reflection across the lines in \eqref{equ:refllines}. There exists a constant $\epsilon_0>0$, depending only on $G$, such that, if $v_1,\ldots v_k\in A$ are contained in some orbit and some ball $B_{\epsilon_0}(v)$ then $v_1,\ldots v_k$ are elements in the orbit of some point group $H_v$ of $G$. 
\end{theorem}
\begin{proof}
Since each $g\in G$ is an isometry and $GD=A$, we need only consider balls with center in the fundamental domain $D$. Consider the vector $v_1 = (1/3,1/3,1/3)$, and let $c_1=\sqrt{1/6}$ so that $B_{c_1}(v_1)$ is tangent to the line $\overleftrightarrow{v_2v_3}$. Since $B_{c_1}(v_1)$ is contained in the region $H_{v_1}D$ and $H_{v_1}$ acts transitively on $H_{v_1}D$, any collection of points in $B_{c_1}(v_1)$ that are $G$-equivalent must be $H_{v_1}$-equivalent. Thus, $G$-equivalence implies $H_{v_1}$ equivalence for any collection of points in a ball of radius $B_{\epsilon}(w)$, so long as $0<\epsilon<\sqrt{1/6}$ and $w\in B_{c_1-\epsilon}(v_1)$. We can make analogous conclusions for $v_2, c_2=\sqrt{1/2}$ and $v_3,c_3=\sqrt{1/8}$. It remains to show $B_{c_1-\epsilon_0}(v_1),B_{c_2-\epsilon_0}(v_2),B_{c_3-\epsilon_0}(v_3)$ cover $D$ for some $\epsilon_0>0$. Consider the barycenter $b=(11/18,5/18,2/18)$ of $D$, and let $d_1,d_2,d_3$ be the distances from $b$ to $v_1,v_2,v_3$ respectively. Explicit calculation shows $d_i<c_i, i =1,2,3$. By convexity, $B_{d_i}(v_i)$ contains the convex hull of $v_i,(v_i+v_{i+1})/2,(v_i+v_{i+2})/2$ and $b$, $i =1,2,3$, indices mod $3$. These $3$ convex regions cover $D$, seen by expressing $w\in D$ as a convex combination of $v_1,v_2,v_3$. Thus $\epsilon_0 = \min\{c_1-d_1,c_2-d_2,c_3-d_3\} = c_1-d_1 =\frac{3-\sqrt{7}}{3\sqrt{6}}$.  The theorem follows. 
\end{proof}
\remark{\Cref{thm:local prop} shows that $G$ acts properly discontinuously on $A$, giving $D$ is has the structure of the orbifold $A/G$ with covering space $A$ \cite[Thm 13.2.1]{Th02}.}. Each open set $U\subset D$ in the relative topology lifts to an open set $GU \subset A$. The identifications unfold $D$ along the lines of reflection in \eqref{equ:refllines}. 

\section{Linear maps on $A/G$}
\label{sec:trimaps}
A function $M:A\to A$ {\emph{preserves re-expression}} if $v\sim w$ implies $Mv\sim Mw$. Any function that preserves re-expression induces a map on $A/G$, hence a triangle map on $\mathcal{T}/S$. We say that a triangle map $f$ is \emph{linear} if $f:D \to D$ is the quotient map of an invertible linear map $M:\R^3\to \R^3$ that leaves $A$ invariant and the restriction of $M$ to $A$ preserves re-expression. We shall henceforth refer to $M$ as an {\emph{angle transition matrix}}, writing ATM for short. Since $M$ leaves $A$ invariant, each column of an ATM sums to $1$. Recall that $A_p$ is the equilateral triangle in $A$ with vertices $e_1,e_2,e_3$. As $M$ is invertible, linear, and leaves $A$ invariant, $MA_p$ is a nonflat triangle in $A$ and the columns of $M$ are the vertices of $MA_p$.  The geometry of the moduli space $A/G$ (\Cref{thm:local prop}) imposes strong conditions on the type of triangle $MA_p$ can be, yielding a classification.\\ 
Our first theorem requires two observations. Define $\Lambda = e_1 + \Z\{e_2-e_1,e_3-e_2\}$ to be the rank two sublattice of $A$ consisting of integer entries. Notice that each $v\in\Lambda$ is met by $6$ lines of \eqref{equ:refllines}. Since $\Lambda = Ge_1$ and $H_{e_1}$ is the only point group from $D$ with order $12$, $|H_v|=12$ if and only if $v\in\Lambda$. Take any point group $H_v$ of $G$ and consider an $H_v$-orbit, $O$. The average $a:=\frac{1}{|O|}\sum_{w\in O}w$ is a fixed point of $H_v$.  If $|H_v|>2$, then $v$ is the only fixed point of $H_v$, so $a = v$. 
\begin{lemma}
\label{thm:lambda}
For any ATM $M$, $M\Lambda\subset \Lambda$. 
\end{lemma}
\begin{proof}
Let $\lambda\in\Lambda$ be arbitrary. Since the point group $H_{\lambda}$ is $G$-conjugate to $H_{e_1}$, $|H_{\lambda}|=12$. Let $\epsilon_0$ be the constant of Theorem \ref{thm:local prop}. As $M$ is linear, $M$ is uniformly continuous. Thus, there exists $\delta > 0$ such that, for all $w_1,w_2\in A$, $|w_1-w_2|<\delta$ implies $|Mw_1-Mw_2|<\epsilon_0$.  Select any orbit $O$ of $H_{\lambda}$ that has $12$ elements $p_1,\ldots,p_{12}$ contained in $B_{\min\{\epsilon_0,\delta\}}(\lambda)$ which average to $\lambda$; such an $O$ exists. As $M$ is invertible, $Mp_1,\ldots,Mp_{12}$ are $12$ distinct points contained in $B_{\epsilon_0}(M\lambda)$. Since $M$ preserves re-expression, $MGp\subset GMp$, so $Mp_1,\ldots,Mp_{12}$ belong to some $G$-orbit. By Theorem  \ref{thm:local prop}, $MO = \{Mp_1,\ldots,Mp_{12}\}$ belongs to an orbit of some point group $H_v$. By the orbit-stabilizer theorem, the cardinality of an $H_v$-orbit is at most $|H_v|$. Since $|H_v|\leq 12$ and $|MO|=12$, $|H_v|=12$. The preceding paragraph shows $v\in\Lambda$, and that the average
$$a=\frac{1}{12}\sum_{i=1}^{12} Mp_i = M\left(\frac{1}{12}\sum_{i=1}^{12} p_i\right) = M\lambda$$ must be equal to $v$. Thus $M\lambda \in \Lambda$, proving the proposition.
\end{proof}
Specilaizing \Cref{thm:lambda} to $e_1,e_2,e_3\in\Lambda$ shows that the entries of $M$ must be integers. Now, let $R\subset A$ be the set of all $w\in A$ with nontrivial point group, meaning $w\in R$ if and only if $w$ is fixed point of at least one reflection in $G$. Thus, $R$ is the union of lines described in \eqref{equ:refllines}. 
\begin{lemma}
\label{thm:edges}
For any ATM $M$, $MR\subset R$.
\end{lemma}
\begin{proof}
The point group of any $w\in A\setminus R$ is trivial. Since $M$ is uniformly continuous and preserves re-expression, we must have $MR\subset R$ by \Cref{thm:local prop}.
\end{proof}
The local point group property (Theorem \ref{thm:local prop}) imposed a strong condition on $M$ by looking at the corners of $A_p$; the next proposition examines the edges of $A_p$. Recall that $v_1 = b=(1/3,1/3,1/3)$ is the barycenter of $A_p$.
\begin{prop}
\label{thm:equil}
The columns of $M$ are vertices of an equilateral triangle in $A$. Moreover, a subgroup of $H_{Mb}\subset G$ permutes the columns of $M$    
\end{prop}
\begin{proof}  
By \Cref{thm:lambda} the vertices of $M A_p$ lie in $\Lambda$. Since the boundary of $A_p$ is a subset of $R$, \Cref{thm:edges} states that each edge of $MA_p$ must lie on one of the $12$ lines of reflection through a vertex of $MA_p$. Taking pairs of edges, the interior angles of $MA_p$ are positive integer multiples of $\pi/6$. Up to permutation, the angles of $MA_p$ are either $(\pi/6,\pi/3,\pi/2),(\pi/3,\pi/3,\pi/3)$, or $(\pi/6,\pi/6,2\pi/3)$. Since a point $v\in A$ on a median of $A_p$ is fixed by the reflection in $G$ that transposes the equal components, $v\in R$. As $MR\subset R$, $Mv\in R$. By linearity, \Cref{thm:edges} shows $MA_p$ cannot have an angle of $\pi/6$. So the vertices of $MA_p$ must form an equilateral triangle. Again by linearity, the medians of $MA_p$ are concurrent at $Mb$. Since $MA_p$ is equilateral, reflections about its medians permute the vertices of $MA_p$, fixing $Mb$. By \Cref{thm:edges}, each median of $MA_p$ lies on a line in $R$, so the medial reflections are elements of $G$. So $H_{Mb}$ contains the subgroup of medial reflections which permute the vertices of $MA_p$.
\end{proof}
\begin{definition}
A $3\times 3$ matrix $M$ is \emph{circulant} if
$$M= \begin{bmatrix}
c_0 & c_1 & c_2 \\
c_2 & c_0 & c_1 \\
c_1 &  c_2 & c_0 \\
\end{bmatrix}$$ and $M$ is \emph{symmetric} if $M^T=M$.
\end{definition}
If $M$ is circulant and symmetric, then $M$ looks like
$$M= \begin{bmatrix}
c_0 & c_1 & c_1 \\
c_1 & c_0 & c_1 \\
c_1 &  c_1 & c_0 \\
\end{bmatrix}.$$
To complete the classification, notice that each element $g\in G$ is the rather uninteresting ATM that reexpresses the angles of a triangle. Since each $g\in G$ corresponds to the identity on $\mathcal{T}/S$, we say $M$, $gM$ are \emph{equivalent} ATM's. As $|H_{b}|=6$, uniform continuity of $M$ and \Cref{thm:local prop} imply $|H_{Mb}|=6$ or $12$.  We are now ready to prove the main Theorem.

\begin{proof}[Proof of \Cref{thm:circsym}]
We draw the boundary of $MA_p$, citing appropriate theorems along the way. By \Cref{thm:lambda}, $Me_1\in\Lambda$. As $MR\subset R$ (\Cref{thm:edges}),  $Me_1-Me_2$ is a integer multiple of one of the $12$ directions 
\[e_1-e_2, e_2-e_3, e_3-e_1, -2e_1+e_2+e_3, e_1-2e_2+e_3, e_1+e_2-2e_3 \] that span lines of reflection through $Me_1$. By  \Cref{thm:equil} $MA_p$ is equilateral whose medians lie in $R$.  Computing the angles between any of the $12$ vectors above, replacing $M$ by $hM$ if necessary $(h\in H_{Mb})$, the boundary directions take one of the two following forms:
\begin{case}
$Me_1-Me_2 = k(e_1-e_2),Me_2-Me_3 = k(e_2-e_3),Me_3-Me_1 = k(e_3-e_1)$. We separate into subcases based on whether $|H_{Mb}|=6$ or $12$. 
\end{case} \begin{subcase}
$|H_{Mb}|=6$. \end{subcase}
Recall that $b$ is the only point in the fundamental domain $D$ with point group of order $6$,  Thus $|H_{Mb}|=6$ implies $Mb\in Gb$, so there exists $g\in G$ such that $gM$ is matrix with $b$ as a fixed point. By \Cref{thm:lambda}, $gMe_1\in\Lambda$, and by \Cref{thm:edges}, the medial line from $gMb=b$ to $gMe_1$ must lie on one of the lines $x-y = 0,x-z=0,y-z=0$ in $R$ through $b$. Thus, $gMe_1$ has integer components (\Cref{thm:lambda}) with two components equal. By \Cref{thm:equil}, $H_{Mb}$ pemutes the columns $gMe_1$ $gMe_2$, $gMe_3$, and $H_{Mb} = H_b$ is the group of permutation matrices in $\R^3$ \eqref{permmat}. Thus there exists $Q\in H_b$ such that $QgM$ is circulant and symmetric, i.e., a Type I matrix \eqref{equ:goodmat}.  
\begin{subcase}
\label{1.2} $|H_{Mb}|=12$. 
\end{subcase} Thus $Mb\in\Lambda$. Since the translation subgroup $\Z\{e_1-e_2,e_2-e_3\}$ acts transitively on $\Lambda$, there exists a \emph{translation} $t\in G$ sending $Mb$ to $e_3$. Consider the subset of $R'\subset R$ of reflection lines given by  \[x-y\in \Z,y-z\in \Z,x-z\in\Z.\]  Recall $w=(1/3,1/3,-2/3)$ and the translation $T_w$ sending $e_3$ to $b$. Notice that $T_{w}$ leaves $A$ invariant but $T_w\notin G$. Nevertheless, straightforward calculations verify that the medians of $MA_p$, by the assumption of column differences, are subsets of $R'$. More straightforward calculations show $tR'=R'$, $T_{w}R'=R'$, and that $T_{w}\Lambda$ is the subset of elements whose point group is order $6$, consisting of triples of third integers $(a/3,b/3,c/3)$ in $A$ with $a\equiv b\equiv c \equiv 1$ mod  $3$. Thus, the vertices $T_wtMe_i$ are elements of $T_w\Lambda$ that lie on the lines $y-z = 0,x-z=0,y-z=0$ in $R'$ through $b$. As $(T_wtM)A_p$ is equilateral,
\[T_{w}tM=\begin{bmatrix}
c_0/3 & c_1/3 & c_1/3 \\
c_1/3 & c_0/3 & c_1/3 \\
c_1/3 &  c_1/3 & c_0/3 \\
\end{bmatrix}\]
yielding \eqref{equ:badmat1}.
\begin{case}
$Me_1-Me_2 =  k(1,1,-2),Me_2-Me_3 =k(-2,1,1),Me_3-Me_1 = k(1,-2,1)$.\end{case} Observe, 
\[Mb = \frac{1}{3}(Me_1+Me_2+Me_3) = \frac{1}{3}(Me_1+Me_1-k(1,1,-2)+Me_1+k(1,-2,1)). \]
Thus $Mb \in \Lambda$, so $|H_{Mb}|=12$. Let $b^{\perp}$ be the orthogonal hyperplane to $b$ in $\R^3$. Direct calculation shows $Me_i-Mb = k(e_{i+2}-e_{i+1})$ (indices modulo 3), thus $Me_i-Mb\in b^{\perp}$.  Similar to subcase \ref{1.2}, there exists a translation $g\in G$ of $A$ with $gMb=e_3$. Since $g$ is given by matrix multiplication that acts as the identity on $(1/3,1/3,1/3)^{\perp}$, we have $gMe_i = gMb +  gk(e_{i+2}-e_{i+1}) = e_3 +  k(e_{i+2}-e_{i+1}),$ or 
\[gM=\begin{bmatrix}
0 & k & -k \\
-k & 0 & k \\
k+1 &  -k+1 & 1 \\
\end{bmatrix}\]
yielding \eqref{equ:badmat2}.
\end{proof}
\begin{cor}
\label{thm:pedallin}
The pedal mapping $P: D\to D$ may be defined where $P(v)$ is flat so that $P$ is linear.
\end{cor}
\begin{proof}
The observations in section \ref{sec:pedal} verify that $P:D\to D$ agrees with the quotient of the Type $1$ ATM  
\[M=\begin{bmatrix}
-1 & 1 & 1 \\
1 & -1 & 1 \\
1 &  1 & -1 \\
\end{bmatrix}\]
for triangles $T$ whose first pedal triangle is not flat. This ATM shows precisely how to define the $P$ when $P([T])$ is flat so that the extended map is linear. 
\end{proof}
The dynamics of linear triangle maps now follow from well known results  (cf \cite[pp 80 156]{KH97}).
\begin{cor}
\label{thm:dynamics}
Let $f:\mathcal{T}/S\to\mathcal{T}/S$, with ATM $M$. the image $MD$ of the fundamental domain $D$ is a $30-60-90$ triangle in $A$ whose edges lie in $R$. Consequently, there exist $|det(M)|$ elements of $GD$ which tile $MD$. The preimages of these $|det(M)|$ triangles form a Markov partition of $D$, making $M$ semiconjugate to a one-sided Bernoulli shift on $|det(M)|$ symbols. Thus $M$ is mixing, and hence, ergodic.   
\end{cor}
See \cite{Al93} for a detailed discussion of a Markon partition for $P$. 
\section{Type I triangle maps}
\label{sec:construct} 
\Cref{thm:circsym} found triangle maps abstractly as angle transition matrices on $A$. We offer one method for constructing triangles with prescribed ATM $M$. We focus on Type I ATM's, following the intuition for the pedal map developed in \cref{sec:pedal}. In this section, we find it easier to work in $A_p$ instead of $D$. Let $M$ be a nonidentity Type I ATM, i.e., $M$ is circulant and symmetric, with integer entries and columns that sum to $1$. Observe that $M$ has one eigenvalue of $1$ with eigenvector $b=(1/3,1/3,1/3)$, and repeated eigenvalue $c_0-c_1= 1-3c_1 \in\Z$ with eigenvectors $e_1-e_2,e_2-e_3$. As $c_1\neq 0$, $M$ is expanding. Thus $M^{-1}$ is a contraction map on $A$, hence $A_p$, where $M^{-1}A_p$ is an equilateral triangle with barycenter $b$. Notice that $M^{-1}$ does not preserve re-expression;, if $v$ is a nearest neighbor to $b$ with $v\sim b$ then $b=Mb \nsim Mv$ because $M$ is a contraction. We seek invertible matrices $M_1\cdots M_k$ such that $M^{-1} = M_1^{-1}\cdots M_k^{-1}$ and each $M_i$ is easily recognized by some triangle construction. Inverting, $M = M_k\cdots M_1$. If we allow for intermediate products $M_l\ldots,M_1$ to not preserve re-expression, then any triangle map with Type I matrix can be constructed as a composition inverse processes. 
The intermediate constructions will be similar to the construction of Hofstadter triangles \cite{Ki92}. Let $T_0=ABC$ be a triangle with shape $(\alpha_0,\beta_0,\gamma_0)\in A_p$ and let $0<r<1$. To construct $T_{-1}$, rotate the line $\overline{AB}$ about $B$ and towards $C$ and rotate the line $\overline{AC}$ about $C$ and towards $B$ until the lines intersect at a point $A'$ in the interior of $T_0$, so that $T_{-1}=A'BC$ has shape $(\alpha_{-1},\beta_{-1},\gamma_{-1})=((1-r)\alpha+(1-r)\beta+\gamma,r\beta,r\gamma)$, or
\[H_{A,r} = 
\begin{bmatrix}
\alpha_{-1}\\
\beta_{-1}\\
\gamma_{-1}
\end{bmatrix}=
\begin{bmatrix}
1 & 1-r & 1-r \\
0 & r & 0 \\
0 &  0 & r \\
\end{bmatrix}
\begin{bmatrix}
\alpha_0\\
\beta_{0}\\
\gamma_{0}
\end{bmatrix}.\]
Similarly, define $H_{B,r}$ and $H_{C,r}$. We call  and $H_{A,r},H_{B,r},H_{C,r}$ \emph{Hofstadter matrices}. If we include the antipedal map $P^{-1}$, we claim that for any Type I ATM $M$, there exists $r_1,r_2,r_3$ such that
\begin{equation}
\label{equ:decomp}
M^{-1} = H_{A,r_1}H_{B,r_2}H_{C,r_3}P^{-1} \text{ or }
M^{-1} = H_{A,r_1}H_{B,r_2}H_{C,r_3}
\end{equation}

To see this, notice that $M^{-1}$ remains circulant and symmetric. Thus $M^{-1}A_p$ is an equilateral triangle with barycenter $b$ and edges parallel to $A_p$. Since the repeated eigenvalue of $P$ is $-2<0$, either $MA_p$ or $MPA_p$ is homothetic to $A_p$ with positive scaling factor $r$ that fixes $b$. notice, $H_{A,r}A_p$ is an equilateral triangle with corners $e_1, (1-r)e_1+re_2,(1-r)e_1+re_3$, so $H_{A,r}$ shrinks $A_p$, fixing $e_1$ and the lines $\overleftrightarrow{e_1e_2}$ and $\overleftrightarrow{e_1e_3}$ with  $H_{A,r}\overleftrightarrow{e_2e_3}$ parallel to $M^{-1}\overleftrightarrow{e_2e_3}$. So choose $r_1$ so that  $H_{A,r_1}\overline{e_2e_3}$ lies on $M^{-1}\overleftrightarrow{e_2e_3}$ (resp $M^{-1}P^{-1}\overleftrightarrow{e_2e_3}$. Notice that for any  $r_2$, $H_{A,r_1}H_{B,r_2}A_p$ will be an equilateral triangle with $H_{A,r_1}H_{B,r_2}\overleftrightarrow{e_2e_3}$ lying on $M^{-1}\overleftrightarrow{e_2e_3}$ (resp $M^{-1}P^{-1}\overleftrightarrow{e_2e_3}$), so choose $r_2$ so that $H_{A,r_1}H_{B,r_2}A_p$ shares a vertex and two edges of $M^{-1}A_p$ (resp $M^{-1}P^{-1}A_p$). Notice that precomposing $H_{A,r_1}H_{B,r_2}$ with $H_{C,r_3}$ fixes this shared vertex, and the desired $r_3$ can be found.\\
Inverting, say $H_{A,r_1}$, construct a triangle $T_1$ from $T_0$ as follows: on a triangle with interior angles $(\alpha,\beta,\gamma)$ rotate the line $\overleftrightarrow{AB}$ about $B$ and away from $C$ by $(1-1/r)\beta$ and rotate the line $\overleftrightarrow{AC}$ about $C$ and away from $B$ by $(1-1/r)\gamma$, calling $A'$ the new intersction point of the rotated lines. If $(1-1/r)(\beta+\gamma)\in\Z$ then the new triangle has a pair of parallel lines and if $1<(1-1/r)(\beta+\gamma)<2$, then $A'$ lies on the other side of $\overleftrightarrow{BC}$ than $A$. Nevertheless we may extend $H_{A,r_1}^{-1}$ by linearity to a linear map defined on $A$. Notice that none of $H_{A,r_1}^{-1}$, $H_{B,r_2}^{-1}$, $H_{A,r_3}^{-1}$ preserve re-expression, since none leave $\Lambda$ invariant. However,

\[M=PH_{C,r_3}^{-1}H_{B,r_2}^{-1}H_{A,r_1}^{-1},\]      
shows $PH_{C,r_3}^{-1}H_{B,r_2}^{-1}H_{A,r_1}^{-1}$ preserves re-expression. 
As an example, consider the Type I matrix 
\[N = \begin{bmatrix}
-3 & 2 & 2 \\
2 & -3 & 2 \\
2 &  2 & -3 \\
\end{bmatrix},\]
with repeated eigenvalue $-5$. $N$ induces a $25\ (=|\det(N)|)$ to $1$ map on $A/G$, with inverse 
\[N^{-1} = \begin{bmatrix}
1/5 & 2/5 & 2/5 \\
2/5 & 1/5 & 2/5 \\
2/5 &  2/5 & 1/5 \\
\end{bmatrix} \]
The proof that outlines the decomposition \eqref{equ:decomp} also provides an algorithm, and  
\begin{align*} 
N^{-1} &= H_{A,4/5}H_{B,3/4}H_{C,2/3}P^{-1}\\ &=\begin{bmatrix}
1 & 1/5 & 1/5 \\
0 & 4/5 & 0 \\
0 &  0 & 4/5 \\
\end{bmatrix}\begin{bmatrix}
3/4 & 0 & 0 \\
1/4 & 1 & 1/4 \\
0 &  0 & 3/4 \\
\end{bmatrix}\begin{bmatrix}
2/3 & 0 & 0 \\
0 & 2/3 & 0 \\
1/3 &  1/3 & 1 \\
\end{bmatrix}\begin{bmatrix}
0 & 1/2 & 1/2 \\
1/2 & 0 & 1/2 \\
1/2 &  1/2 & 0 \\
\end{bmatrix}.
\end{align*}
\begin{figure}
\label{fig:twentyfivetoone}
\begin{center}
\scalebox{.4}{\includegraphics{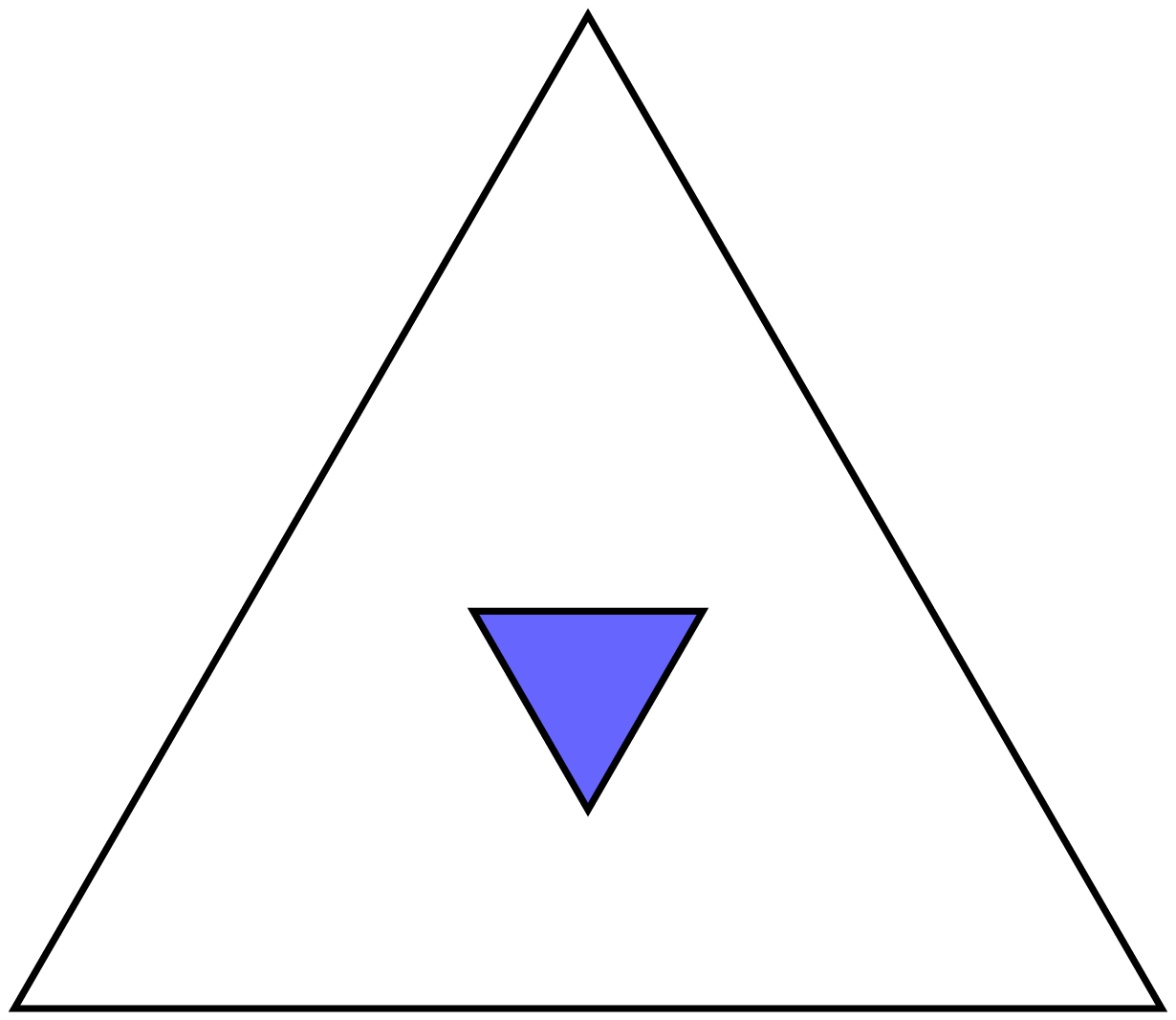}}\scalebox{.4}{\includegraphics{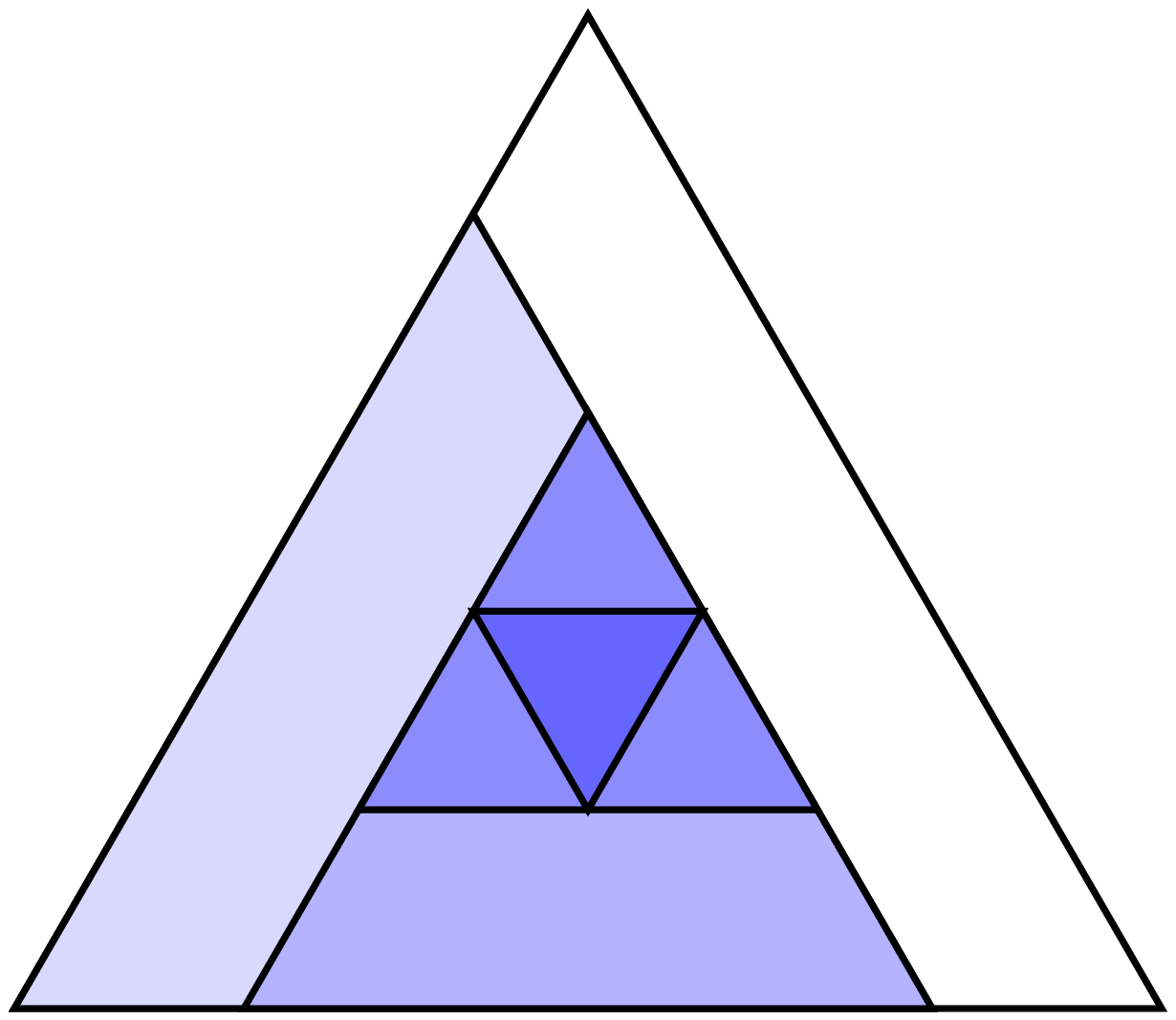}}
\end{center}
\caption{The figure on the left is $N^{-1}A_p$ shaded in blue in $A_p$ (white). The figure on the right is the effect of the composition $H_{A,4/5}H_{B,3/4}H_{C,2/3}P$ on $A_p$. The four nested shaded triangles show the composition.} 
\end{figure}
The effect of this composition on $A_p$ is shown in Figure \ref{fig:twentyfivetoone}
\section{Future Research}
The construction in section \cref{sec:construct} is not canonical. Indeed, the Hofstatdter matrices $H_{A,r},H_{B,r},H_{C,r}$ are conjugate by permutation matrices, and \[ N^{-1} = H_{C,4/5}H_{B,3/4}H_{A,2/3}P^{-1}\] is another way of decomposing $N^{-1}$. We wonder whether there is a canonical construction whose inverse is $N^{-1}$, similar to the pedal mapping. We guess that its discoverer would be rewarded with a wonderful picture. We also wonder about connections between triangle maps and constructable numbers.\\
Obviously, one could relax the assumption of linearity, to find, e.g., continuous maps $f:A\to A$ that preserve re-expression. One might investigate linear $n$-gon maps (though convexity might be an issue). \cite{DHZ03} examined convergent polygon constructions and our inversion approach could unveil new constructions. In \cite{Al93}, the author posed a question about the dynamics of pedal tetrahedron and we include linear tetrahedron constructions to the list of open questions. 

\section{Acknowledgment.}
 Much of this work was completed as a summer reseach project that was funded through a Dean of Faculty grant at Amherst College. The second author thanks Rob Benedetto and Peter Connor for many helpful comments on a draft version of this paper, and Dan Velleman, Michael Ching and Rob Benedetto for many useful discussions.

\end{document}